\documentclass[10pt,a4paper]{article}
\usepackage{amsmath,amsfonts,amsthm}
\usepackage{graphicx}
\usepackage[all]{xy}
\usepackage{psfrag}

\newtheorem{thm}{Theorem}[section]

\newtheorem{cor}[thm]{Corollary}
\newtheorem{lem}[thm]{Lemma}
\newtheorem{prop}[thm]{Proposition}
\theoremstyle{definition}
\newtheorem{defn}[thm]{Definition}
\theoremstyle{remark}

\numberwithin{equation}{section}

\renewcommand{\>}{\rangle}

\begin{document}

\author{L. Ciobanu and A. Ould Houcine}
\title{The monomorphism problem in free groups}
\date{}
\maketitle

\begin{abstract}
Let $F$ be a free group of finite rank. We say that the monomorphism problem in $F$ is decidable if for any two elements $u$ and $v$ in $F$, there is an algorithm that determines whether there exists a monomorphism of $F$ that sends $u$ to $v$. In this paper we show that the monomorphism problem is decidable and we provide an effective algorithm that solves the problem. 
\medskip



\end{abstract}

\section{Introduction}

Let $F$ be a free group of finite rank. Given two elements $u$ and $v$ in $F$, one can formulate the following natural decision problems. Is there an algorithm that determines whether there is a homomorphism $\phi : F \rightarrow F$ such that $\phi(u)=v$? If we require $\phi$ to be an automorphism, we call the decision problem the automorphism problem in free groups, and similarly define the endomorphism and the monomorphism problems. In the case of the automorphism problem, the question has been answered positively by Whitehead in 1936 \cite{Whitehead36} and his algorithm is one of the most important and useful tools when computing in free groups and beyond. Answering the endomorphism problem is equivalent to solving an equation in free groups in which the variables and constants appear on different sides of the equality. Thus the solvability of the endomorphism problem in free groups is a consequence of Makanin's algorithm \cite{M}.   

Answering the monomorphism problem is equivalent to deciding if the following infinite system
$$
u(X_1, \dots, X_n)=v(x_1, \dots, x_n) \wedge  \bigwedge_{w \in W}w(X_1, \dots, X_n) \neq 1, \leqno (1)
$$
has a solution in $F=\<x_1, \dots, x_n\>$, where $W$ is the set of nontrivial reduced words in $F$. The existence of a solution of the infinite system in (1) can be expressed as an existential  sentence  in the language $\mathcal  L_{\omega_1\omega}$, where $\mathcal L$ denotes the usual language of groups and $\mathcal  L_{\omega_1\omega}$ is built, roughly speaking, from $\mathcal L$ by allowing  infinite conjunctions and disjunctions of formulas. Thus the solvability of the monomorphism problem in free groups can be seen as the  problem of deciding the truth of a subset of the set  of existential sentences of $\mathcal  L_{\omega_1\omega}$ in $F$.  It is worth pointing out that for $n \geq 3$  the infinite system in (1) is not equivalent to a finite subsystem and thus the problem cannot be reduced to the decidability of the existential theory of $F$.

In this paper we give a positive answer to the monomorphism problem in free groups. We   provide an algorithm that is polynomial in the lengths of $u$ and $v$,  except for parts that involve the Whitehead algorithm, which is conjectured to be polynomial, but this has not yet been proven. 
\section{Preliminaries}

Let $F_n$ be a free group of rank $n\geq 2$ with generating set $A=\{ x_1,\ldots , x_n \}$, viewed as the
fundamental group of the wedge of $n$ circles. This naturally leads to working with graphs. All graphs considered here
are going to be oriented and finite (unless otherwise stated).

Let $H$ be a finitely generated subgroup of rank $m$ of the free group $F_n$, and let $X_H$ be the corresponding
covering space of the wedge of $n$ circles (infinite except when $H$ has finite index in $F_n$). That is, vertices of
$X_H$ are cosets, $V(X_H)=\{ Hx \mid x\in F_n \}$, and edges are of the form $(Hx,a)$ going from $Hx$ to $Hxa$, for all
$x\in F_n$ and $a\in A$. Note that $X_H$ is an $A$-labeled oriented graph, with a distinguished basepoint $*=H1$, and
with every vertex being the initial vertex (and the terminal vertex as well) of exactly $n$ edges, labeled by the $n$
symbols in $A$ (see~\cite{kap-mya-stallingsfoldings} for more details).

The {\it core} of $H$, denoted $C_H$, is the smallest subgraph of $X_H$ containing the basepoint $*$, and having
fundamental group $H$. So all vertices in $C_H$ have degree at least two except possibly $*$ and, since $H$ is
finitely generated, $C_H$ is a finite graph. Like $X_H$, the graph $C_H$ is an $A$-labeled oriented graph, with every
vertex being the initial vertex (and the terminal vertex) of at most $n$ edges, labeled by pairwise different letters
in $A$.

We will later make use of some particular type of graphs, which we call topological graphs.

\begin{defn}
A {\it topological graph} of rank $g$ is a finite graph with a distinguished vertex $*$ in which all vertices have degree at least $3$  except possibly $*$,  and whose fundamental group is the free group $F_g$ of rank $g$. Let $Top(g)$ be the set of topological graphs of rank $g$.  

\end{defn}

It is worth pointing out that, if $H$ is a finitely generated subgroup of $F_n$, then one can associate to $H$ a topological graph obtained from $C_H$ by deleting all vertices of $C_H$ of degree 2 (except $*$). The next lemma provides some basic information about the set $Top(g)$.

\begin{lem} \label{topgraph}
The set $Top(g)$ is finite, and the number of edges in a topological graph of rank $g$ is at most $3g-1$.
\end{lem}

\begin{proof}
Let $\Gamma$ be a topological graph of rank $g$. The Euler characteristic formula gives $|E(\Gamma)|-|V(\Gamma)|+1=g$, where $|E(\Gamma)|$ and $|V(\Gamma)|$ are the number of edges and vertices of $\Gamma$, respectively. This can be rewritten as $$\frac{\sum_{v \in V(\Gamma)}\deg(v)}{2} - |V(\Gamma)|+1=\frac{\sum_{v \in V(\Gamma)}(\deg(v)-2)}{2} +1=g.$$ 
Set $V(\Gamma)^*=V(\Gamma)\setminus \{*\}$.  Then  
$$
2(g-1)=\sum_{v \in V(\Gamma)^*}(\deg(v)-2)+(\deg(*)-2),
$$ so $\Gamma$ has at most $2g$ vertices since $\deg(v)-2\geq 1$ for all $v \in V(\Gamma)^*$ and $\deg(*)-2 \geq -1$. A consequence of this is that $|E(\Gamma)|=g+|V(\Gamma)|-1 \leq g+2g -1=3g-1.$

Since $Top(g)$ is a set of graphs with a bounded number of vertices and edges, this set is  finite.  
\end{proof}


In our main result we will need the concept of a Nielsen-reduced set, which is defined as follows. Let  $|u|$ denote the length of a word $u$ in $F$ with respect to the basis $A$. A subset $U \subseteq  F \setminus \{1\}$ is called Nielsen-reduced if for any $v_1, v_2, v_3 \in U^{\pm 1}$ the following conditions hold:
\begin{enumerate}
\item $v_1 v_2 \neq 1$ implies $|v_1 v_2| \geq |v_1|, |v_2|$,
\item $v_1 v_2 \neq 1$ and $v_2 v_3 \neq 1$ implies $|v_1 v_2 v_3| > |v_1|-|v_2|+|v_3|$. 
\end{enumerate}

Such a set has some very desirable properties. Let $H$ be a subgroup of $F$ generated by a Nielsen-reduced set $U$. Then
\begin{enumerate}
\item $H$ is free with $U$ as a basis \cite[Proposition 2.5, Ch I]{LyndonSchupp77},
\item if $w \in H$ has the form $w=u_1 u_2 \dots u_t$ where each $u_i\in U^{\pm 1}$, $u_i u_{i+1} \neq 1$ and $t \geq 1$, then $|w|\geq t$ and $|w|\geq |u_i|$ for any $1\leq i \leq t$ \cite[Proposition 2.13, Ch I]{LyndonSchupp77}. 
\end{enumerate}

It is well-known that any subgroup of $F$ has a Neilsen-reduced basis. We can conclude the previous two remarks with the following corollary.

\begin{cor}  \label{cor} Let $F$ be a free group of rank $n \geq 2$. Then any finitely generated subgroup $H$ of $F$ has a basis $B=\{b_1, \dots, b_m\}$ such that for any reduced nontrivial word $w$ on $B$  one has $|w| \geq |b|$ for any $b \in B$ which appears in the reduced form of $w$.  \qed 
\end{cor}

\smallskip
The previous corollary states that, in some sense, $H$ has a basis consisting of ``minimal"  elements, and an example of such a basis could be a Nielsen-reduced one. Under some natural conditions, there is also a generalisation of it to the context of valuated groups \cite[Lemma 4.2]{ould}. 

\smallskip
We will also  need the following  lemma. 

\begin{lem}\label{lem1}
Let $F$ be a free group of rank $n \geq 2$. If $H$ is a subgroup of rank $m<n$, then  there are elements $c_1, \dots, c_{n-m}$ such that the subgroup generated by $H$ and $\{c_1, \dots, c_{n-m}\}$ is  of rank $n$. 
\end{lem}

\proof 

It is sufficient to show that $H$ is contained in a subgroup of infinite rank. If it is not the case, by \cite[Proposition 3.15, Ch I]{LyndonSchupp77}, $H$ has finite index. But in that case, since $|F:H|(rk(F)-1)=(rk(H)-1)$,   we have $rk(H) \geq n$. \qed

\section{The main result}

\begin{thm} \label{thm} Let $F$ be a free group of rank $n \geq 2$. Then there is an algorithm which decides, given   $u, v \in F$, whether there exists a monomorphism $f : F \rightarrow F$ such that $f(u)=v$. 
\end{thm}

Theorem \ref{thm} relies on the following key proposition. 

\begin{prop}\label{prop1}
Let $F$ be a free group on $x_1, \dots, x_n$.  Let $u=u(x_1, \cdots, x_n)$ be a reduced word and let $v \in F$. The following properties are equivalent:

\begin{enumerate}
\item[(1)] there is a monomorphism $f : F \rightarrow F$ such that $f(u)=v$; 
 \item[(2)]  there exist $m$,  $1 \leq m \leq n$,  and elements $b_1, \dots, b_m$ in $F$ such that:

\smallskip
$(i)$ the group $H=\<b_1, \cdots, b_m\>$ is free of rank $m$,  $|b_i| \leq |v|$,  $v \in H$, 

$(ii)$ there exists an automorphism $h$ of  $H*\<y_1, \dots,y_{n-m}|\>$ such that  $$h(u(b_1, \dots, b_m, y_1, \dots, y_{n-m}))=v. $$

\end{enumerate}
\end{prop}

\proof  $(1)\Rightarrow (2)$.  Let  $f :F \rightarrow F$  be a monomorphism such that  $f(u)=v$.  By Corollary \ref{cor}, the subgroup $f(F)$ has a basis  $B=\{b_1, \dots, b_n\}$ such that for any reduced nontrivial word $w$ on $B$  one has $|w| \geq |b|$ for any $b \in B$ which appears in the reduced form of $w$. 

Since $v \in f(F)$, $v$ can be written on $B$, and thus there exists, relabeling the $b_i's$ if necessary, $1 \leq m \leq n$ such that $v \in H=\<b_1, \dots, b_m\>$ and $|b_i| \leq |v|$ for $1 \leq i \leq m$. We have 
$$
f(u(x_1, \cdots, x_n))=u(f(x_1), \dots, f(x_n))=v.
$$

Since $\{f(x_1), \dots, f(x_n)\}$ is a basis of $f(F)$,  by defining $h :f(F) \rightarrow f(F)$, by 
$$
h(b_i)=f(x_i), 
$$
$h$ is an automorphism of $f(F)$ and 
$$
h(u(b_1, \cdots, b_n))=u(h(b_1), \dots, h(b_n))=u(f(x_1), \dots, f(x_n))=v.
$$
Thus there exists an automorphism $h$ of $f(F)$ such that $h(u(b_1, \cdots, b_n))=v$.   Since $H*\<y_1, \dots,y_{n-m}|\>$ is isomorphic to $f(F)$ and $v \in H$,  the same conclusion apply   also  for $H*\<y_1, \dots,y_{n-m}|\>$; that is there exists an automorphism $h$ of $H*\<y_1, \dots,y_{n-m}|\>$ such that $h(u(b_1, \dots, b_m, y_1, \dots, y_{n-m}))=v$.

$(2)\Rightarrow (1)$. Suppose first that $m=n$. By defining $f : F \rightarrow F$ as
$$
f(x_i)=h(b_i),
$$
we have 
$$
f(u(x_1, \dots, x_n))=u(h(b_1), \dots, u(b_n))=v, 
$$
and thus $f$ is a monomorphism such that $f(u)=v$. 

Suppose now that $m<n$. By Lemma \ref{lem1},  there exist $c_1, \dots c_{n-m} \in F$ such that the subgroup  $K=\<b_1, \dots, b_m, c_1, \dots, c_{n-m}\>$ is free of rank $n$. 

Therefore there exists an automorphism $h$ 
of $K$ such that
$$h(u(b_1, \dots, b_m, c_1, \dots, c_{n-m}))=v.$$

By defining  $f : F \rightarrow F$ as 
$$
f(x_i)=h(b_i)   \hbox{ for }  1 \leq i \leq m, \quad f(x_{m+j})=h(c_j), \hbox{ for }  1 \leq j \leq n-m, 
$$
we have 
$$
f(u(x_1, \dots, x_n))=v, 
$$
and thus $f$ is a monomorphism such that $f(u)=v$.  \qed

\bigskip
The following result is well-known, and it is shown in \cite{Touikan} that the mentioned algorithm can be performed in time less than $O(n\log(n))$, where $n$ is the sum of the lengths of the words in $U$.

\begin{prop} \label{prop2}Let $F$ be a free group. Then there is an algorithm which, given a finite subset $U$ of $F$ and an element $v$ of $F$, gives the rank of $\<U\>$ and decides if $v \in \<U\>$ and if so it gives a word $w$ on $U$ such that $v=w$.
\end{prop}

\bigskip
\noindent \textbf{Proof of Theorem \ref{thm}}

The algorithm that solves the monomorphism problem is the following. 

\noindent \textbf{Input}: $u(x_1, \dots, x_n)$ and $v$ in $F$ 

\noindent \textbf{Output}: YES or NO (there exists a monomorphism sending $u$ to $v$ or not)
\begin{description}
\item (1) List all sets $U_1, \dots, U_p$ such that $|U_i| \leq n$ and for any $x \in U_i$,  $|x|\leq |v|$.
\item (2) For $i=1$ to $p$ do
\begin{description}
\item (3) Determine the rank of $K_i=\<U_i\>$ by using Proposition \ref{prop2}. 
\item (4) If $rk(K_i) \neq |U_i|$ then go to $i+1$ else determine if $v \in K_i$.
\item (5) If $v \not \in K_i$ then go to $i+1$ else find a word on $U_i$ such that $v=w$.
\item (6) Let $m=|U_i|$. 
\item (7) By using Whitehead's algorithm, determine if there exists an automorphism  $h$ of  $K*\<y_1, \dots,y_{n-m}|\>$ such that  $$h(u(b_1, \dots, b_m, y_1, \dots, y_{n-m}))=w.$$
\item (8) If $h$ exists then by Proposition \ref{prop1} return YES, else go to $i+1$. 
\end{description}
\item (9) If there is no positive output, then return NO. 
\end{description}
\qed
 
\noindent \textbf{Remark}. The algorithm can in fact also produce $f(x_1), \dots, f(x_n)$.

\section{A polynomial time algorithm}

The algorithm in the proof of Theorem \ref{thm} consists of two main parts.

(I) First one finds all tuples $\{b_1, \dots, b_m\}$ such that $v \in \langle b_1, \dots, b_m \rangle$, $|b_i|\leq |v|$ and $\{b_1, \dots, b_m\}$ freely generate $\langle b_1, \dots, b_m \rangle$. Then one finds all words $w \in \langle b_1, \dots, b_m \rangle$ such that $w(b_1, \dots, b_m)=v$.

(II) The second part consists of applying Whitehead's algorithm to the words $u$ and $w$, for each $w$ found in (I). That is, one needs to check whether there is an automorphism $$\alpha: \langle b_1, b_2, \dots, b_m, y_1, \dots y_{n-m}\rangle \rightarrow \langle b_1, b_2, \dots, b_m, y_1, \dots y_{n-m}\rangle$$ such that $\alpha(u(b_1, b_2, \dots, b_m, y_1, \dots y_{n-m}))=w(b_1, \dots, b_m).$

As presented in the proof of Theorem \ref{thm}, line (1) is exponential in the length of $v$. Part (II), Whitehead's algorithm,

 is known to be at most exponential, but conjectured to be polynomial in the lengths of the words \cite{MS, RVW}. 

Here we provide an alternate algorithm for part (I) that can replace lines (1) - (5) and can be performed in time polynomial in $|v|$. Instead of producing all possible tuples of words of bounded length, among which there is a lot of redundancy (in the sense that many of them do not freely generate a subgroup, and many identical subgroups occur), we generate directly candidates for the subgroups that satisfy the properties described by Proposition \ref{prop1} part 2(i). The subgroups that we produce are called \textit{test subgroups} and are defined in 4.4. Roughly speaking, we generate topological graphs of rank up to $n$ whose edges we thereafter label with words in $F$ such that a tuple $\{b_1, \dots, b_m\}$ that labels the generators of the fundamental group of one of these graphs satisfies the following:  $v \in H=\langle b_1, \dots, b_m \rangle$, $|b_i|\leq |v|$, each $b_i$ appears in $v$, and the core graph $C_H$ has a minimal number of edges. Simultaneously we get the set of words $w$ such that $w(b_1, \dots, b_m)=v$.

For the remainder of the paper, we will use the term \textit{arc} in a graph for paths whose interior vertices have degree 2. Each core graph corresponds to a topological graph whose edges are labeled by words in the free group in a way such that no foldings can be performed after the labeling. The fact that $v$ can be read as a loop in such a graph is equivalent to writing $v$ as as the concatenation of the labels on a subset of the edges of the graph. This leads to the following lemma, which counts the number of ways in which $v$ can be achieved as such a concatenation.

\begin{lem} \label{prop3}
Let $v$ be a word in $F_n$ with $|v|=l$ and $l, k\geq 1$. Then the number of morphisms $\phi :F_k \rightarrow F_n$ with the properties:
\begin{enumerate}
\item[(a)] there exists a word $w \in F_k$ such that $\phi(w)=v$, and
 \item[(b)] no cancellation occurs between the images $\{ \phi(x_i) \}$ when forming $v$, 

\end{enumerate}
 is less than $p(l)$, where $p(l)$ is a polynomial of degree $2k$. Here $x_1, \dots, x_k$ denote the generators of $F_k$.
\end{lem}

\begin{proof}

We can determine the $\phi(x_i)$ one at a time. Since there is no cancellation when forming $v$, each of the $\phi(x_i)$ appear as subwords in $v$. Since any word of length $l$ has $l^2$ possible subwords, there are $(l^2)^k$ possible values for the tuples $(\phi(x_1), \dots, \phi(x_k))$. If we also allow for $\phi(x_i)=1$, we get an upper bound of $(l+1)^{2k}$ for the number of morphisms in the lemma. 
\end{proof}

The next two propositions show that, given a topological graph $\Gamma$, it is sufficient to consider a  number that is polynomial in $|v|$, of ways to label $\Gamma$, in order to generate all test subgroups $H$ (see Definition 4.4) with property 2(i) from Proposition \ref{prop1}. (A priory the number of colorings is exponential, since each arc could be labeled by a word of length smaller than $|v|$.) Since the number of topological graphs that we use is given by $\sum_{i\leq n}|Top(i)|$, which is a function of $n$, these results provide the polynomial time algorithm we are interested in.

\begin{defn} Let $H$ be a subgroup of  $F_n$, $v \in H$, and $C_H$ the core graph of $H$.  An arc $c$ of $C_H$ is said to be \textit{visible} (relative to $v$)  if it traversed when we read $v$ along arcs of $C_H$. The other arcs are called \it{invisible}.  

\end{defn}

\begin{prop} \label{invisible}  If condition (2) of Proposition \ref{prop1} is satisfied, then we can choose  $H$ in such a way that any invisible arc,  relative to $v$,  is of length at most $3$. 
\end{prop}

\begin{proof}

Let  $H=\<b_1, \dots, b_m\>$ satisfying (2) of Proposition \ref{prop1} such that $C_H$ has a minimal number of edges.  We claim that any invisible arc, relative to $v$,  in $C_H$ has length at most $3$.   Suppose by contradiction that there exists an invisible arc $c$  of length greater than 3. To obtain the desired contradiction, we will show that there is a subgroup $H'$ that meets condition (2)  of Proposition \ref{prop1} and such that  $H'$  has fewer edges  than $H$.

Write  $c=e_1e_2 \cdots e_t$ with $t \geq 4$ (where $e_i$ denotes an edge).  Replace $c$ by $e_1\alpha e_t$,  where $\alpha$ is a new edge. We claim that there is a labeling of $\alpha$ in such a way that the obtained graph is reduced.  Let $a$ (resp. $b$)  to be the label of $e_1$ (resp. $e_t$).  If $a \neq b^{-1}$ then we  label  $\alpha$ by  $a$.  If $a=b^{-1}$ then we pick  $a' \in A$ with $a' \neq   b, b^{-1}$ and we label $\alpha$ by $a'$. Hence we get a reduced graph $\Gamma'$ as claimed.  

Let $H'$ to be the subgroup of $F$ whose core graph is $\Gamma'$.  Using the Euler characteristic formula,  a simple count shows that $H'$ has the same rank as $H$. Let $b_i'$ to be  the element  of $H'$ obtained by replacing each occurence of $c$ in $b_i$  by  its new labelling in $\Gamma'$.  Clearly $|b_i'|   \leq v$ and $H'=\<b_1', \dots, b_{m}'\>$. Since $c$ is an invisible arc,  $v \in H'$. 

To get a contradiction and to finish the  proof, it is enough to show that  there exists a word $w(x_1, \cdots, x_m)$ such that  $v=w(b_1, \cdots, b_m)=w(b_1', \cdots, b_m')$. 

Let $x$ be a new variable and let $L=F*\<x|\>$.  Let $\Gamma$ be the $A\cup \{x\}$-labeled graph obtained by lalbeling $c$ by $x$ (i.e., deleting all interior vertices of $c$ and we label $c$ by $x$).  Clearly  $\Gamma$ is reduced and  a simple count shows that the rank of $\Gamma$ is equal to the rank of $C_H$.  Let $G$ be the group whose core graph is $\Gamma$.  Let $g_i$ to be  the element  of $G$ obtained by replacing each occurence of $c$ in $b_i$  by  its new labelling in $\Gamma$. Proceeding as above, we conclude that $G=\<g_1, \cdots, g_m\>$ and $v \in G$.  Let $w(x_1, \cdots, x_n)$ such that $v=w(g_1, \cdots, g_m)$. 

Let $f$ (resp. $f'$)  to be the   homomorphism from $L$ to $F$ which fixes every element of $F$ and which sends $x$ to the label of $c$ in $C_H$ (resp. $C_{H'}$).  We have    $v=f(v)= w(f(g_1), \dots, f(g_m))$ and $v=f'(v)= w(f'(g_1), \dots, f'(g_m))$.   But $f(g_i)=b_i$ and $f'(g_i)=b_i'$. Hence, we obtain the desired conclusion. 
\end{proof}

\begin{defn} A subgroup $H$ of $F_n$  such that $v\in H$ and $H$ satisfies the condition in Proposition \ref{invisible} is called a \textit{test subgroup} (relative to $v$).  

\end{defn}

\begin{prop} \label{prop4}
The number of test subgroups $H$, relative to $v$, of rank $g$ corresponding to a fixed topological graph $\Gamma$ is bounded above by $A(g)p(l)n^{6g-3}$, where $l=|v|$, $A(g)$ is a function that depends on $g$ only and $p(l)$ is a polynomial of degree $2g$.
\end{prop}

\begin{proof}
Since $\Gamma$ has rank $g$, Lemma \ref{topgraph} implies that the number of arcs of $\Gamma$ is $|E(\Gamma)|\leq 3g-1$. As the total number of arcs in $\Gamma$ depends on $g$, we will denote $|E(\Gamma)|$ also by $t(g)$, where `t' stands for \textit{total number of arcs}. Fix a numbering for the arcs of $\Gamma$, say $e_1, \dots, e_{t(g)}$.

Now let $k=t(g)$. We start by choosing a morphism $\phi:F_k \rightarrow F_n$ with the properties given in Lemma \ref{prop3}. Let $\phi(x_i)=X_i$ for all $1\leq i \leq k$. We `color'  or `label' $e_1, \dots, e_{t(g)}$ with $X_1$ up to $X_k$ and $B$, where $B$ stands for `blank.' Let  $L(\Gamma)=\{X_1, \dots, X_k, B\}$. In case some of the $X_i$ are the empty word, we remove those $X_i$ from $L(\Gamma)$ and only use the remaining labels. Each coloring of the graph should correspond to a surjection from the set $E(\Gamma)=\{e_1, \dots, e_{t(g)}\}$ to $L(\Gamma)=\{X_1, \dots, X_k, B\}$, with the additional requirement that $L(\Gamma)$ does not contain any $X_i=1$, and that if $X_i\neq X_j$ for all $1\leq i\neq j \leq k$, then $L(\Gamma)$ does not contain $B$. Notice that for each morphism $\phi$ the set $L(\Gamma)$ will have at most $k=t(g)$ elements.

Thus the number of valid labelings of $\Gamma$ with $\phi$ is bounded above by the number of surjections from the set $E(\Gamma)$ to $L(\Gamma)$. Let the number of these surjections be $A(g)$. If we let $|L(\Gamma)|=m$, then $A(g)$ is the second Sterling number $S(k,m)$. (If $|E(\Gamma)|=|L(\Gamma)|$, then $A(g)=k!$.) Since $k$ and $m$ depend on $g$, $A(g)$ is a function that depends on $g$ only.
 
Finally, we need to color the `blank' or invisible arcs, that is, the edges labeled by $B$ so far and which by Proposition \ref{invisible} can be labeled with at most three letters without loss of generality.  Let $b(g)$ be the number of the invisible arcs. 

Since $v$ has to contain each generator of the group, we can assume that it passes through at least $g$ arcs of $\Gamma$. That is, the number of invisible arcs should be at most $t(g)-g\leq 3g-1-g=2g-1$, and all the colorings which result in having more that $2g-1$ invisible arcs can be discarded. Then the labeling of these arcs can be done in $(2n(2n-1)(2n-1))^{b(g)} \leq 8n^{3(2g-1)}$ ways, where $n$ is the rank of the free group. 

Since the number of morphisms in Lemma \ref{prop3} is bounded above by $p(l) \leq (l+1)^{2k}$, we get that the number of test subgroups $H$ of rank $g$ is less than $A(g)p(l)n^{3(2g-1)}$.
\end{proof}

\subsection{Generating test subgroups}
Part (I) of the algorithm now follows the outline:

\noindent \textbf{Input}:  word $v$ of length $l$, free group of rank $n$. 

\noindent \textbf{Output}: word $w$ such that $w(h_1, \dots, h_m)=v$ for some $h_i \in F$ such that $\langle h_1, \dots, h_m \rangle$ is a test subgroup. 
\begin{description}
\item (1) For $g=1$ to $n$ generate the topological graphs of rank $g$. 
\item (2) For $\Gamma \in Top(g)$ do
\begin{description}
\item (3) Let $k:=|E(\Gamma)|$ Generate basis $h_1, \dots, h_g$. 
\item (4) Generate morphisms $F_k \rightarrow F_n$ as in Lemma \ref{prop3}
\item (5) For each morphism above and for each basis: 
\item (6) Let $X_j=\phi(x_j)$ and label arcs of $\Gamma$ with $X_1, \dots X_k, B$ as in the proof of Proposition \ref{prop4}
\item (7) For each labeling check if $v$ can be read as a loop at $\star$ in $\Gamma$
\item (8) If yes, chek if $|h_i| \leq |v|$
\item (9) If yes, then return $w$, where $w$ is the word in $h_1, \dots, h_g$ corresponding to the loop found in (7)
\item (10) Go to next morphism
\item (11) Go to next basis
\end{description}
\item (12) Go to next $\Gamma$
\end{description}

The only lines that involve the length of $v$ are (4), and (7) and (8). Line (4) will be performed $p(l)$ times within the second \textit{for} loop, as shown in Lemma \ref{prop3}. Lines (7) and (8) can be performed in a time almost linear in $l$, as shown in \cite{Touikan}. Analyzing the algorithm above, we see that the number of operations performed is bounded above by the sum: 
\begin{equation*}\tag{$\star$}
\sum_{g=1}^n |Top(g)|A(g)(l+1)^{6g}n^{3(2g-1)} = O(l^{6n})
\end{equation*}
\section{Conclusions}

The algorithm that we provided for the monomorphism problem, with the exception of part two of the Whitehead algorithm, is polynomial in the lengths of the words $u$ and $v$. However, the constants involved in the time complexity in ($\star$) are exponential in the rank of the group $F$, and thus we cannot claim that our algorithm is a practical one. For free groups of small ranks the constants are manageable due to the fact that the number of topological graphs is small, and the degree of the polynomial $p(l)$ is also very small. In particular, if $F$ has rank $2$, it is known that the Whitehead algorithm has polynomial complexity, which leads us to the following corollary.

\begin{cor}
The monomorphism problem in the free group of rank $2$ has a time complexity that is polynomial in the lengths of the words $u$ and $v$. 
\end{cor}

The rank $2$ free group is possibly the only one in which the three related decision problems, the endomorphism \cite{LC}, monomorphism and automorphism \cite{MS} problem, can be solved in a time polynomial in the lengths of the words involved.

A consequence of our method  is also the following. \begin{cor}
The multiple monomorphism problem in free groups is solvable. That is, for any $r \geq 1$ and any two tuples of words $(u_1, \dots, u_r)$ and $(v_1,\dots, v_r)$ in $F$, one can decide whether there is an injective homomorphism  $f:F \rightarrow F$ such that $f(u_i)=v_i$ for all $1 \leq i \leq r$.
\end{cor}
\begin{proof} 
The proof of Proposition \ref{prop1} easily extends to the tuple case. The existence of a monomorphism $f$ is equivalent to the existence of a subgroup $H=\<b_1, \cdots, b_m\>$, free of rank $m\leq n$,  such that $|b_i| \leq |v_j|$ and $v_j \in H$ (condition 2(i)), and the existence of an automorphism $h$ such that $$h(u_j(b_1, \dots, b_m, y_1, \dots, y_{n-m}))=v_j$$ for all $1\leq i \leq m$, $1 \leq j \leq r$ (condition 2(ii)).

One then follows the outline of the algorithm given in Theorem \ref{thm}. One finds subgroups $H=\<b_1, \cdots, b_m\>$ as above by enumerating tuples (line (1) of the algorithm) and then checking whether they satisfy all the required properties (lines (3), (4) and (5)). Then one writes each $v_i$ as a word $w_i$ in the generators $b_1, \cdots, b_m$ (line (5)) and applies the Whitehead algorithm for tuples in the group  $H*\<y_1, \dots,y_{n-m}|\>$ in order to determine whether there exists an automorphism $h$ such that  $h(u_j(b_1, \dots, b_m, y_1, \dots, y_{n-m}))=v_j$ for all $1 \leq j \leq r$.
\end{proof}

\section*{Acknowledgments}
The first-named author was partially supported by the SNF (Switzerland) through project
number 200020-113199 and by the Marie Curie Reintegration Grant 230889.

\bibliographystyle{plain}

\bigskip
\noindent Laura CIOBANU, D\'epartement de math\`ematiques,
Universit\'e de Fribourg,
Chemin du Mus\'ee 23,
CH-1700 Fribourg P\'erolles,
Switzerland.\\
\textit{E-mail}: \textrm{laura.ciobanu@unifr.ch}\\
\\
\noindent Abderezak OULD HOUCINE, 

\noindent Universit\'e de Mons-Hainaut,
Institut de Math�matique,
B\^atiment �Le Pentagone�,
Avenue du Champ de Mars 6,
B-7000 Mons,
Belgique. 

\medskip
\noindent Universit\'e de Lyon; Universit\'e Lyon 1;
INSA de Lyon, F-69621;
Ecole Centrale de Lyon;
CNRS, UMR5208, Institut Camille Jordan,
43 blvd du 11 novembre 1918,
F-69622 Villeurbanne-Cedex, France. 
\textit{E-mail}: \textrm{ould@math.univ-lyon1.fr}

\end{document}